\documentclass[11pt]{amsart} 

\usepackage{amsmath}
\usepackage{amsthm}
\usepackage{amssymb}
\usepackage{mathrsfs}
\usepackage{mathtools}

\usepackage{xcolor}
\usepackage{framed}
\usepackage{tikz}
\usepackage{float}
\usepackage{hyperref}
\usepackage{url}

\definecolor{shadecolor}{gray}{0.9}

\newtheorem{theorem}{Theorem}[section]
\newtheorem{lemma}[theorem]{Lemma}
\newtheorem{proposition}[theorem]{Proposition}
\newtheorem{corollary}[theorem]{Corollary}

\theoremstyle{definition}
\newtheorem{definition}[theorem]{Definition}
\newtheorem{remark}[theorem]{Remark}
\newtheorem{example}[theorem]{Example}

\numberwithin{equation}{section}

\newcommand{\cQ}{\mathcal{Q}}
\newcommand{\R}{\mathbb{R}}
\newcommand{\N}{\mathbb{N}}
\newcommand{\Z}{\mathbb{Z}}
\newcommand{\Q}{\mathbb{Q}}
\newcommand{\cX}{\mathcal{X}}
\newcommand{\cC}{\mathcal{C}} 
\newcommand{\cPx}{\mathcal{P}} 

\newcommand{\mC}{\Omega} 
\renewcommand{\r}{\mathcal{R}} 
\newcommand{\f}{f} 
\renewcommand{\S}{\mathcal{L}} 
\newcommand{\V}{\mathscr{V}} 
\newcommand{\SE}{S \mkern-2mu E} %
\newcommand{\SO}{S \mkern-2mu O} %
\newcommand{\ucp}{uniform-cap property}

\DeclareMathOperator{\vol}{v}
\DeclareMathOperator{\Stab}{Stab}
\DeclareMathOperator{\Orbit}{Orbit}

\begin{document}

	\title[Hypersphere-Based Restricting Conditions]{Hypersphere-Based Restricting Conditions for Colorings of the Euclidean Space}

	\author{Gabriel Istrate}
	\address{
		Faculty of Mathematics and Computer Science\\ 
		University of Bucharest\\
		Bucharest, Romania}
	\email{gabriel.istrate@unibuc.ro} 
	\author{Catalin Zara}
	\address{Department of Mathematics\\ 
		University of Massachusetts Boston\\
		Boston, MA, United States of America}
	\email{catalin.zara@umb.edu}
	
	\date{}
	

	\begin{abstract}
		We study colorings of the Euclidean space constrained by \emph{hypersphere forcing conditions}: if an admissible hypersphere, \(S_r(p)\), centered at a point \(p\) and of radius \(r\) contains a monochromatic set of points satisfying a certain property \(\cPx\), then the center of the hypersphere must have that color. These forcing conditions may be restricted in applicability to a specific set of hyperspheres \(S_r(p)\). 
		
		For cardinality-based forcing conditions we prove a general theorem: for countably many colors and any uncountable set of admissible radii \(\r\), such a coloring is locally monochromatic on any admissible center set \(\mC \subseteq \R^n\) (hence constant, for  connected \(\mC\)). 
		
		For rigid geometric properties (simplex shape, edge-length, volume constraints) we show that forcing conditions alone are insufficient without regularity assumptions. Our main result shows that for colorings satisfying a certain Baire regularity condition rigid geometric properties enforce local monochromaticity and, in the presence of a certain \emph{``uniform cap" condition}, global monochromaticity. 
		
		Applications include dichotomies for edge-length and volume constraints in terms of \(\inf(\S)\) and \(\inf(\V)\), and a comeagerness criterion in the ``all edges in \(\S\)'' regime.
	\end{abstract}
	
	\maketitle

	\tableofcontents

	\section{Introduction}
	\label{sec:intro}
	A central theme of \emph{Euclidean Ramsey theory} \cite{Graham} is that simple local restrictions on combinatorial-geometric configurations, most prominently those arising from colorings of \(\R^n\), often force global structure.

	In this paper we deal with one concrete scenario of this type. Specifically,  we investigate colorings of the \(n-\)dimensional Euclidean space \(\R^n\) (\(n\geqslant 2\)) in which the color of a point \(p\) is constrained by monochromatic configurations lying on (hyper)spheres\footnote{For convenience, unless specified otherwise, we will use the term \emph{sphere} to mean \emph{hypersphere}.}  centered at \(p\).
	The concrete manifestation of the ``local-to-global forcing'' phenomenon is the following: if sufficiently many (or sufficiently structured) monochromatic configurations occur on many admissible spheres around \(p\), then the forcing rule propagates color information and can compel the coloring to be locally constant, or even globally monochromatic.

	\subsection{Hypersphere forcing conditions}
	
	A \emph{coloring} of a topological space \(\cX\) is a function \(\f \colon \cX\to \cC\) from \(\cX\) to a nonempty set \(\cC\).
	For each \(c\in \cC\), the set \(\f^{-1}(c)\) is the \emph{color class} of \(c\).
	We say that \(\f\) is \emph{monochromatic} if it is constant on \(\cX\), \emph{locally monochromatic} if every point has an open neighborhood on which \(\f\) is constant, and \emph{somewhere locally monochromatic} if \(\f\) is constant on some nonempty open subset of \(\cX\).
	Throughout, \(\cX\) will be \(\R^n\) with its usual topology, together with induced subspace topologies.
	
	We will work in the following general framework: 
	Fix \(n\geqslant 2\), a set of colors \(\cC\subseteq \N\), a nonempty set \(\mC \subseteq\R^n\) of \emph{admissible centers}, and a nonempty set \(\r \subseteq(0,\infty)\) of \emph{admissible radii}.
	For \(p\in\R^n\) and \(r>0\) let \(S_r(p)\) denote the sphere of radius \(r\) centered at \(p\); we call it \emph{admissible} if \(p\in\mC\) and \(r \in \r\).
	Let \(\cPx\) be a fixed property that subsets of \(\R^n\) may or may not satisfy.
	We study colorings \(\f \colon \R^n \to \cC\) satisfying the condition
	
	\begin{shaded}
		\(\cQ_{\cC,\cPx,\mC,\r}\):
		For any color \(c \in \cC\) and any admissible sphere \(S_r(p)\),
		if the set \(\f^{-1}(c) \cap S_r(p)\) has a subset with property \(\cPx\),
		then \(\f(p) = c\).
	\end{shaded}
	
	Thus, if an admissible sphere contains a monochromatic configuration of type \(\cPx\), then the center is forced to have that monochromatic color.
	The central question of this paper is:
	
	\begin{shaded} 
		\textbf{When does \(\cQ_{\cC,\cPx,\mC,\r}\) force \(\f\) to be locally monochromatic on \(\mC\), or even monochromatic?}
	\end{shaded}
	
	This formulation explicitly decouples the \emph{geometric constraints} (the property \(\cPx\)) from the \emph{largeness/regularity assumptions} (the admissible centers and radii, and later, regularity assumptions on color classes).
	
	\subsection{A compelling, robust example: cardinality forcing}
	
	Our study was inspired by the following problem, slightly paraphrased,
	which appeared as Problem~B1 on the 86th W. L. Putnam Mathematical Competition 
	in December~2025, see \cite{Putnam}.
	
	\smallskip
	
	\emph{Suppose that each point in the plane is colored either \(1\) (red) or \(2\) (green),
		subject to the following condition.
		If a circle contains three points of the same color \(c \in \{1,2\}\),
		then the center of the circle also has color \(c\).
		Show that all points in the plane have the same color.}
	
	\smallskip
	
	In the framework described above, we have \(n=2\) (two-dimensional space), \(\cC = \{1,2\}\) (two colors),
	\(\mC = \R^2\) (all points are admissible centers),
	and \(\r = (0,\infty)\) (all radii are admissible).
	The property \(\cPx\) is that ``\emph{the set has cardinality at least \(3\)}.''
	
	One of the solutions presented in \cite{Putnam} considers a slightly more general
	setting that allows an arbitrary finite number \({X}\) of colors and requires that the
	intersection contain at least a fixed number \({Y}\) of points in order to force
	the conclusion about the color of the center.
	In this case \(\cC = \{1,\ldots,{X}\}\), and the property \(\cPx\) is that ``\emph{the set has
		cardinality at least \({Y}\)}.''
	We denote the corresponding condition \(\cQ_{\cC,\cPx,\mC,\r}\) by
	\(\cQ_{{X},{Y},\mC,\r}\).
	The conclusion remains the same: 
	If all circles are admissible, that is, if \(\mC = \R^2\) and \(\r = (0,\infty)\),
	then any coloring \(\f\) satisfying \(\cQ_{{X},{Y},\mC,\r}\) is monochromatic.

	Section~\ref{sec:cardinality} considers this type of restriction in greater generality.
	We replace the ambient space \(\R^2\) by a general \(\R^n\), \(n \geqslant 2\), allow countably many\footnote{we use the term \emph{countable} to mean the cardinality of a subset, finite or infinite, of \(\N\)} colors and countably many forcing thresholds \({Y}\), allow \(\mC \subseteq\R^n\) to be an arbitrary nonempty admissible-center set, and restrict radii to an arbitrary \emph{uncountable} set \(\r\).
	Theorem~\ref{thm2} shows that in this cardinality regime the forcing principle is already extremely demanding:
	any coloring satisfying \(\cQ_{{X}, {Y}, \mC, \r}\) must be \emph{locally monochromatic on \(\mC\)}, and hence constant on connected \(\mC\).
	
	\subsection{Rigid geometric forcing properties often require regularity}
	
	In the remainder of the paper we turn to geometric properties \(\cPx\) invariant under rigid motions, most notably, properties defined by the vertices of simplices with prescribed shape, edge-length constraints, or volume constraints.
	
	We show that for such properties forcing behaves differently, compared to cardinality constraints: pathological non-constant colorings, vacuously satisfying such forcing conditions can often be constructed, sometimes under assumptions outside Zermelo-Fraenkel theory. For such geometric forcing rules, \emph{additional regularity assumptions on color classes are necessary} if one wants the forcing condition to imply the existence of somewhere locally monochromatic open structure. For such colorings, we  describe an enforcing  mechanism below.  
	
	\subsection{The main mechanism for geometric forcing conditions: Baire-category forcing at small radii}
	
	Beginning in Section~4 we impose a Baire regularity hypothesis by working with \emph{somewhere comeager} colorings \(\f \colon \R^n \to \cC\): these are colorings for which some color class is comeager on a nonempty open set.
	We deal with forcing rules where all points are admissible centers (\(\mC = \R^n\)) and the set \(\r\) of admissible radii is \emph{comeager near \(0\)}. Our main result, 
	Theorem~\ref{thm:general_formQ}, gives sufficient conditions on
	properties \(\cPx\) under which the colorings under consideration necessarily produce a monochromatic open set (hence enforcing somewhere local monochromaticity), and under a stronger ``uniform-cap'' hypothesis force global monochromaticity. In Theorem~\ref{thm:rigid_Q} we apply this general result to the special case of \emph{rigid geometric properties}, using the representations of spheres as homogeneous spaces under the transitive action of the group of rigid motions.
	
	These two results are the conceptual engine behind our applications.
	
	\subsection{Applications: simplex shapes, edge-lengths, and volumes}
	
	Sections~\ref{sec:IT-ET-RT}--~\ref{sec:volumes} apply the general framework developed in Theorems~\ref{thm:general_formQ} and~\ref{thm:rigid_Q} to properties \(\cPx\) defined by classes of simplices. Those conditions are given in terms of 
	shapes (Section~\ref{sec:IT-ET-RT}), edge lengths (Section~\ref{sec:edge_length}), 
	or volumes (Section~\ref{sec:volumes}).
	
	\begin{itemize}
		\item \textbf{Simplex shapes (Section~\ref{sec:IT-ET-RT}
			).}
		For isosceles, regular (equilateral), and right simplices, we prove that the forcing rules under consideration always yield a monochromatic open set (Theorem~\ref{thm:BPdisk}). The obstruction to global monochromaticity (in the examples we study) is the fact that, for those simplices,
		the circumcenter lies in the convex hull of the vertices.
		
		\item \textbf{Edge-length constraints (Section~\ref{sec:edge_length}).}
		Given \(\S \subset(0,\infty)\), we study forcing rules that trigger when a monochromatic simplex has at least \(k\) edge lengths in \(\S\).
		When \(k<N=\binom{n+1}{2}\), Theorem~\ref{thm:infL} gives a sharp dichotomy:
		such forcing yields monochromaticity for all colorings exactly when \(\inf(\S)=0\).
		When \(k=N\) (all edges constrained), the behavior changes:
		\(\inf(L\S)=0\) no longer suffices (Example~\ref{exmp:not_inf}), and a sufficient hypothesis is that \(\S\) is comeager near \(0\) (Theorem~\ref{thm:comeagerL}).
		
		\item \textbf{Volume constraints (Section~\ref{sec:volumes}).}
		Given \(\V \subset(0,\infty)\), forcing based on the existence of a monochromatic simplex with volume in \(\V\) again has a sharp threshold:
		Theorem~\ref{thm:infL_vol} shows that monochromaticity for all such colorings holds exactly when \(\inf(\V)=0\).
	\end{itemize}
	
	\subsection{Organization of the Paper}
	
	Section~\ref{sec:cardinality} proves the cardinality-forcing theorem (Theorem~\ref{thm2}).
	Section~\ref{sec:rigid_properties} records facts about rigid motions and introduces rigid geometric properties defined via simplices.
	Section~4 develops the Baire-category framework for somewhere comeager colorings and proves the main mechanism (Theorems~\ref{thm:general_formQ} and~\ref{thm:rigid_Q}).
	Sections~\ref{sec:IT-ET-RT}--~\ref{sec:volumes} apply this mechanism to forcing rules defined by simplex shapes, edge lengths, and volumes.

	\section{Cardinality Conditions}\label{sec:cardinality}
	
	In this section we show that the conclusion of the Putnam competition problem
	remains valid in more general settings.
	These include variations in the finite dimension of the space, the cardinality of the set of colors,
	the minimum number of same-color points on a circle required to force the color
	of the center,
	the cardinality of the set of admissible radii,
	and the choice of admissible centers.
	
	In the general framework, let \({Y} \leqslant \aleph_0\) be a fixed cardinality.
	We consider the following property.
	
	\begin{shaded}
		\(\cPx_{({Y})}\):
		A set has this property if its cardinality is at least \({Y}\).
	\end{shaded}
	
	We refer to this property simply as \(({Y})\).
	
	\begin{theorem}\label{thm2}
		Let \(\r \subset (0,\infty)\) be an uncountable set of admissible radii,
		let \(\cC\) be a set of colors of cardinality \({X} \leqslant \aleph_0\),
		let \({Y} \leqslant \aleph_0\) be a fixed cardinality,
		and let \(\mC \subset \R^n\) be a nonempty set of admissible centers.
		Let \(\f \colon \R^n \to \cC\) be a coloring of the space \(\R^n\) satisfying the following
		condition:
		\begin{shaded}
			\(\cQ_{{X},{Y},\mC,\r}\):
			For every color \(c \in \cC\) and every admissible sphere \(S_r(p)\),
			if the set \(\f^{-1}(c) \cap S_r(p)\) has cardinality at least \({Y}\),
			then \(\f(p) = c\).
		\end{shaded}
		Let \(\f|_{\mC} \colon \mC \to \cC\) denote the restriction of \(\f\) to the induced
		topological space \(\mC\).
		Then \(\f|_{\mC}\) is locally monochromatic.
		If \(\mC\) is connected, then \(\f|_{\mC}\) is monochromatic.
	\end{theorem}
	
	\begin{proof}
		Since the uncountable set
		\[
		\r = \bigcup_{m \in \mathbb{Z}} \bigl([2^{m-1}, 2^m] \cap \r \bigr)
		\]
		is a countable union, at least one of these subsets must be uncountable.
		Hence, there exists \(a = 2^{m-1} > 0\) such that
		\([a,2a] \cap \r\) is uncountable.
		
		Let \(p \in \mC\) be an admissible center and let \(B = B_{2a}(p)\).
		If \(B \cap \mC = \{p\}\), then \(\{p\}\) is open in \(\mC\),
		and \(\f\) is locally monochromatic at \(p\).
		
		Otherwise, assume that \(B \cap \mC \neq \{p\}\).
		Choose \(q \in B \cap \mC\) with \(q \neq p\),
		and let \(d = d(p,q) \in (0,2a)\) denote the distance between \(p\) and \(q\).
		Choose \(m \in \mathbb{N}\) such that \(d > a/m = \epsilon\).
		Since the uncountable set
		\[
		[a,2a] \cap \r
		= \bigcup_{k=1}^m \bigl([a+(k-1)\epsilon, a+k\epsilon] \cap \r \bigr)
		\]
		is a finite union, at least one of the subsets must be uncountable.
		Therefore, there exists \(b \in [a,2a)\) such that
		\(\r \cap (b, b+\epsilon)\) is uncountable.
		
		Suppose that \(n=2\), hence, that we are coloring the plane \(\R^2\).
		
		Let \(I \subset \r \cap (b, b+\epsilon)\) be a countable subset.
		For each \(r \in I\), the circle \(C_r(p)= S_{r}(p)\) is admissible.
		For any color \(c \neq \f(p)\), the set
		\(\f^{-1}(c) \cap C_r(p)\) must have cardinality strictly less than \({Y}\),
		otherwise \(\f(p) = c\).
		Hence, the set
		\[
		A = \bigcup_{r \in I} \bigcup_{c \in \cC \setminus \{\f(p)\}}
		\bigl(\f^{-1}(c) \cap C_r(p)\bigr)
		\]
		has cardinality at most \({X} \cdot {Y} \leqslant \aleph_0\),
		and is therefore countable.
		
		Let
		\[
		D = \{ d(q,T) \mid T \in A \}
		\]
		be the set of distances from \(q\) to points of \(A\).
		Then \(D\) is also countable.
		Since \(\r \cap (b,b+\epsilon)\) is uncountable,
		there exists an admissible radius
		\[
		R \in (\r \cap (b,b+\epsilon)) \setminus D .
		\]
		
		Consider the circle \(C_R(q)\).
		For every \(r \in I\), we have
		\[
		|R - r| < \epsilon < d(p,q) < 2a \leqslant 2b <  R + r ,
		\]
		so the circles \(C_R(q)\) and \(C_r(p)\) intersect in exactly two points.
		Because \(C_R(q) \cap A = \emptyset\),
		all such intersection points have color \(\f(p)\).
		Since \(I\) is countable, it follows that
		\[
		\bigl| \f^{-1}(\f(p)) \cap C_R(q) \bigr| \geqslant 2 \aleph_0 = \aleph_0 \geqslant {Y} ,
		\]
		which forces \(\f(q) = \f(p)\).
		
		Suppose now that \(n \geqslant 3\). Let \(r, R \in  \r \cap (b, b+\epsilon)\). The admissible spheres  \(S_r(p)\) and \(S_R(q)\) intersect in an \((n-2)-\)sphere \(S\), and \(S\) has uncountable many points. As a subset of \(S_r(p)\), the sphere \(S\) can have at most countably many points of color other than \(f(p)\), hence, \(S\) has uncountable many points of color \(f(p)\).  Since \(S\) is also a subset of \(S_R(q)\), it follows that \(f(q) = f(p)\).
		
		Thus, for every \(p \in \mC\), there exists an open ball \(B \subset \R^n\)
		such that \(\f(q) = \f(p)\) for all \(q \in B \cap \mC\).
		Since \(B \cap \mC\) is open in \(\mC\),
		the restriction \(\f|_{\mC}\) is locally monochromatic.
		
		If \(\mC\) is connected, then every locally constant function on \(\mC\)
		is constant.
	\end{proof}
	
	\medskip
	
	The next example shows that if \(\mC\) is not connected,
	then \(\f|_{\mC}\) may be locally monochromatic without being globally
	monochromatic.
	
	\begin{example}
		Let \(\cC = \{1,2\}\), let \({Y} \leqslant \aleph_0\),
		let \(\r = (0,1)\), which is uncountable,
		and let
		\[
		\mC = B_1((-2,0)) \cup B_1((2,0)).
		\]
		Define \(\f\) to be constantly \(1\) on \(B_2((-2,0))\),
		constantly \(2\) on \(B_2((2,0))\),
		and arbitrary elsewhere in the plane.
		Then \(\f\) satisfies \(\cQ_{{X},{Y},\mC,\r}\),
		but takes different values on the two connected components of \(\mC\).
		Hence, \(\f|_{\mC}\) is not monochromatic.
		\qed
	\end{example}
	
	The following example shows that even when \(\mC\) is connected and
	\(\f|_{\mC}\) is monochromatic,
	the coloring \(\f \colon \R^2 \to \cC\) need not be somewhere locally monochromatic.
	
	\begin{example}\label{exmp:non_rational}
		Let \(\Q^2\) denote the set of rational points in the plane \(\R^2\),
		and let \(\Omega = \R^2 \setminus \Q^2\).
		Define \(\f \colon \R^2 \to \{0,1\}\) by
		\[
		\f(x) =
		\begin{cases}
			1, & x \in \Q^2,\\
			0, & x \in \Omega.
		\end{cases}
		\]
		For any \({X} \geqslant 2\), \({Y} \geqslant 3\), and \(\r \subset (0,\infty)\),
		the coloring \(\f\) satisfies \(\cQ_{{X},{Y},\mC,\r}\).
		Indeed, if a circle contains at least three rational points,
		then its center is also rational.\footnote{Incidentally, that was problem B1 on the 2008 Putnam Competition.}
		Consequently, if \(p \in \Omega\),
		then for every \(r \in \r\), \(\f^{-1}(1) \cap C_r(p)\) has at most two points and \(\f^{-1}(0) \cap C_r(p)\)  is infinite,
		so \(\f(p) = 0\).
		
		The induced topological space \(\mC\) is connected,
		in fact path-connected,
		and \(\f|_{\mC}\) is constant, consistent with Theorem~\ref{thm2}.
		However, both color classes \(\f^{-1}(1) = \Q^2\) and \(\f^{-1}(0) = \Omega\) are dense in \(\R^2\),
		so \(\f\) is not constant on any nonempty open disk in \(\R^2\).
		
		Note that if \(p \in \Q^2\) and \(r > 0\), then \(\f(p) = 1\) and
		\(\f^{-1}(0) \cap C_r(p)\) is infinite.
		This does not contradict \(\cQ_{{X},{Y},\mC,\r}\),
		since the condition applies only to circles with centers in \(\mC =  \R^2 \setminus \Q^2\).
		\qed
	\end{example}

	\medskip

	\section{Rigid Geometric Properties}\label{sec:rigid_properties}
	In later sections we will study properties \(\cPx\) based on geometric conditions that are invariant under isometries, 
	and in this section we summarize several useful results - for more details, see \cite{Lee} or \cite{Warner}.
	
	Every \emph{rigid motion} (orientation-preserving isometry) of the Euclidean space
	\( \R^n \) can be written as the composition of a translation and a rigid motion that fixes the
	origin \( O \). Moreover, every rigid motion that fixes the
	origin restricts to a rigid motion of the unit sphere
	\[
	S^{n-1}=S_1(O)\subset \R^n .
	\]
	
	The group of rigid motions of the unit sphere \( S^{n-1} \) is the
	\emph{special orthogonal} Lie group \( \SO(n) \), and the group of rigid motions
	of \( \R^n\) is the \emph{special Euclidean} group
	\[
	\SE(n)\simeq \SO(n)\ltimes \R^n,
	\]
	a semidirect product of \( \SO(n) \) and the Abelian group \( \R^n\).
	
	The standard left action of \( \SE(n) \) on \( \R^n \),
	\[
	\Phi \colon \SE(n)\times \R^n \to \R^n,
	\qquad (a,p)\longmapsto a\cdot p := a(p),
	\]
	is smooth and transitive. For every point \( p\in \R^n \), the \emph{stabilizer}
	of \( p \) is the closed Lie subgroup
	\[
	G=\Stab(p)=\{a\in \SE(n)\mid a\cdot p=p\}\simeq \SO(n).
	\]
	
	For every sphere \( S=S_R(p) \) centered at \( p \), the action of \( \SE(n) \)
	on \( \R^n \) induces a smooth and transitive action
	\[
	G\times S \to S .
	\]
	
	\begin{lemma}\label{lem:cont_open}
		Let \( q\in S \) and define
		\[
		\phi \colon G\to S, \qquad \phi(a)=a\cdot q .
		\]
		Then \( \phi \) is continuous, open, and surjective.
	\end{lemma}
	
	\begin{proof}
		The \( G \)-orbit of \( q \) is
		\[
		\phi(G)=\Orbit_G(q)=\{a\cdot q\mid a\in G\}=S,
		\]
		so \( \phi \) is surjective. The stabilizer of \( q \) in \( G \) is
		\[
		\phi^{-1}(q)=\Stab_G(q)=\{a\in G\mid a\cdot q=q\}\simeq \SO(n-1).
		\]
		The natural map
		\[
		G/\Stab_G(q)\longrightarrow \Orbit_G(q)=S,
		\qquad [a]\longmapsto a\cdot q,
		\]
		is a diffeomorphism. Under this identification, \( \phi \) coincides with the
		quotient projection \( G\to G/\Stab_G(q) \). Hence \( \phi \) is continuous,
		open, and surjective.
	\end{proof}
	
	\begin{definition}
		A property \(\cPx\) of sets in \(\R^n\) is a \emph{rigid geometric property} if for every set \(\Gamma \subset \R^n\) that has property \(\cPx\) 
		and rigid motion \({a} \in \SE(n)\), any set that contains \({a}\cdot \Gamma\) has property \(\cPx\).
	\end{definition}
	
	The rigid geometric properties \(\cPx\) that we will consider will be based on \emph{simplices}, and we introduce some terminology at this point - see \cite{Grunbaum} for details.
	
	\begin{definition}
		An \emph{\(m-\)simplex} \(\sigma = \sigma^m\) in \(\R^n\) \((n\geqslant m)\) is the convex hull of \(m+1\) affinely independent points, which are called the \emph{vertices} of \(\sigma^m\). An \(n-\)simplex in \(\R^n\) is called a \emph{full-dimensional simplex}.
	\end{definition}
	
	For example, a 1-simplex is a segment and a 2-simplex is a (solid) triangle. A triangle is a full-dimensional simplex in \(\R^2\), but not in \(\R^3\). The notions of equilateral, isosceles, and right triangles extend to higher dimensional simplices. A simplex is:
	\begin{enumerate}
		\item a \emph{regular simplex} (or \emph{equilateral}) if the Euclidean distances between any two distinct vertices are all equal.
		\item an \emph{isosceles simplex} if there exists a vertex \(p\) such that all the distances from \(p\) to the other vertices of the simplex are equal - such a vertex \(p\) is called an \emph{apex} of the isosceles simplex.
		\item a \emph{right simplex} if there exists a vertex \(p\) such that the lines joining \(p\) with the other vertices are mutually orthogonal - such a vertex \(p\) is called an \emph{apex} of the right simplex.
	\end{enumerate}
	
	For every full-dimensional simplex \(\sigma^n\) in \(\R^n\), there exists a unique sphere \(S_R(p) \subset \R^n\) that contains all the vertices of the simplex. That sphere is called the \emph{circumsphere}, the center \(p\) is the \emph{circumcenter}, and the radius \(R\) is the \emph{circumradius}. For full-dimensional regular simplices and for right simplices in the two-dimensional space, the circumcenter is within the simplex, but that need not be the case for isosceles simplices, and is not the case for right simplices in dimensions greater that two. 
	
	The property of a simplex in \(\R^n\) to be equilateral, isosceles, or right are invariant under rigid motions of \(\R^n\). An \(m-\)simplex \(\sigma\) in \(\R^n\) has a nonzero \(m-\)volume \({\vol}_m(\sigma)\), and the \(m-\)volume is also invariant under rigid motions of \(\R^n\). Examples of rigid geometric properties that we will consider include:
	\begin{itemize}
		\begin{shaded}
			\item[\(\cPx_{(IS^m)}\):] A set in \(\R^n\) has this property if it contains the vertices of an isosceles \(m-\)simplex.
		\end{shaded}
		
		\begin{shaded}
			\item[\(\cPx_{(ES^m)}\):] A set in \(\R^n\) has this property if it contains the vertices of a regular \(m-\)simplex.
		\end{shaded}
		
		\begin{shaded}
			\item[\(\cPx_{(RS^m)}\):] A set in \(\R^n\) has this property if it contains the vertices of a right \(m-\)simplex.
		\end{shaded}
		
		\begin{shaded}
			\item[\(\cPx_{(\vol_m, \V)}\):] A set in \(\R^n\) has this property if it contains the vertices of an \(m-\)simplex with \(m-\)volume in a fixed set \(\V\) of admissible values.
		\end{shaded}
		
		\begin{shaded}
			\item[\(\cPx_{k \S N}\):] A set in \(\R^n\) has this property if it contains the vertices of an \(n-\)simplex with at least \(k\) of the \(N=\binom{n+1}{2}\) edge lengths in a fixed set \(\S\) of admissible values.
		\end{shaded}
		
	\end{itemize}
	
	\medskip
	
	\section{Somewhere Comeager Colorings}\label{sec:somewhere_comeager}
	
	In the next sections we study colorings satisfying conditions \(\cQ\) based on
	properties \(\cPx\) involving simplices of specific shapes.
	
	Ceder \cite{Ceder} showed that the plane can be decomposed into countably many
	sets, none of which contains the vertices of an equilateral triangle.
	Assuming the Continuum Hypothesis, Davies \cite{Davies} (\(n=2\)), and Kunen \cite{Kunen} (all \(n\)),  
	constructed partitions of \(\R^n\) into countably many sets, none of which contains the vertices of an
	isosceles triangle.
	Erd\H{o}s and Komj\'ath \cite{Erdos} obtained the same conclusion for \(n=2\) under Martin’s
	Axiom, and Schmerl \cite{Schmerl2} proved the same result for all \(n\), without additional set-theoretic hypotheses.
	Erd\H{o}s and Komj\'ath \cite{Erdos}, with a proof completed in
	Bursics--Komj\'ath \cite{Bursics} and an alternative approach given in
	\cite{Schmerl3}, showed that assuming the Continuum Hypothesis, the plane can be
	colored with countably many colors so that no color class contains the vertices
	of a right triangle.
	
	In all of these cases, coloring each set of the decomposition with a different
	color yields colorings using countably many colors that admit no monochromatic
	triangles of the specified type.
	For each such coloring, the corresponding conditions \(\cQ\) are logically
	satisfied.
	However, since every nonempty open ball contains isosceles, equilateral, and
	right triangles, no monochromatic open ball can exist.
	Thus, without additional regularity assumptions on the color classes, these
	conditions cannot imply the existence of a monochromatic nonempty open ball.
	
	The regularity condition that we impose on colorings of the space \(\R^n\) is that they be \emph{somewhere comeager colorings} - see Definition~\ref{def:somewhere_comeager} below. For self-containment, we include some basic definitions, properties, and results related to Baire Category here - for more details, see, for example, \cite{Oxtoby}, \cite{Kechris}.
	
	A topological space \(\cX\) is \emph{second-countable} if it has a countable basis - this will be the case in our applications.  A set \(A \subset \cX\) is \emph{nowhere-dense} if \(int(\overline{A})=\emptyset\) - that is, if the interior of the closure is empty. A set is \emph{meager} if it can be written as the countable union of nowhere-dense sets, and otherwise is \emph{nonmeager}. A set \(A\) is \emph{comeager} if its complement \(\cX \setminus A\) is meager.  A set \(A\) \emph{has the Baire property} if there is an open set \(U\) such that the symmetric difference \(M = A \Delta U \) is meager - equivalently, if \(A = M\Delta U\) is the symmetric difference of a meager set and an open set.
	
	A nonempty subset \(\mathscr{U} \subset \cX\) is a topological space with the induced topology. A set \(A\subset \cX\) is \emph{comeager in \(\mathscr{U}\)}  if \(A\cap \mathscr{U}\) is comeager in \(\mathscr{U}\). In particular, a set \(S \subset (0,\infty) \) is \emph{comeager near 0} if there exists \(\epsilon > 0\) such that \(S\) is comeager in \((0,\epsilon)\).
	
	A set \(A \subset \cX\) is \emph{somewhere comeager} if there exists a nonempty open set \(U \subset \cX\) such that \(A\) is comeager in \(U\). If \(A\subset \R^{n}\) has the Baire property, then \(A\) is nonmeager if and only if \(A\) is somewhere comeager. 
	
	\begin{definition}\label{def:somewhere_comeager}
		A coloring \(\f \colon \R^n \to \cC\) is called a \emph{somewhere comeager coloring} if there exists a color \(c\in \cC\) such that the color class \(\f^{-1}(c)\) is somewhere comeager in \(\R^n\).
	\end{definition}
	
	\begin{proposition}
		If \(\f \colon \R^n \to \N\) is a coloring such that all color classes have the Baire property, then \(\f\) is a somewhere comeager coloring.
	\end{proposition}
	
	\begin{proof}
		Since \(\N\) is countable and \(\R^n\) is nonmeager, there exists a color \(c\) such that the color class \(\f^{-1}(c)\) is nonmeager. Since \(\f^{-1}(c)\) is nonmeager and has the Baire property, it follows that \(\f^{-1}(c)\) is somewhere comeager.
	\end{proof}

	We will use the following consequence of the Kuratowski-Ulam theorem.
	
	\begin{lemma}\label{lem:polarKU}
		Let \(p \in \R^n\) and \(R > 0\). If \(A \subset B_R(p)\) is comeager in \(B_R(p)\), then the set
		\[
		G_p = \{ r \in (0,R) \mid A \text{ is comeager in } S_r(p)\footnote{for the induced topology on \(S_r(p)\).} \}
		\]
		is comeager in \((0,R)\). In particular, \(G_p\) is infinite.
	\end{lemma}
	
	\begin{proof}
		Consider the punctured open ball
		\[
		D = \{ p + r q \mid 0 < r < R,\ q \in S^{n-1} \}.
		\]
		Let \(\Phi \colon (0,R) \times S^{n-1} \to D\),
		defined by \(\Phi(r,q) = p + r q\). Then \(\Phi\) is a homeomorphism, and 
		since \(A \setminus \{p\}\) is comeager in \(D\),
		its inverse image  \( \Phi^{-1}(A \setminus \{p\})\) is comeager in
		\((0,R) \times S^{n-1}\).
		By the Kuratowski-Ulam theorem, the set
		\[
		\{ r \in (0,R) \mid \{ q \mid (r,q) \in \Phi^{-1}(A \setminus \{p\} \} \text{ is comeager in } S^{n-1} \}
		\]
		is comeager in \((0,R)\). Under the homeomorphism \(\Phi\), this set corresponds exactly to \(G_p\),
		which completes the proof.
	\end{proof}

	For the remainder of this paper we consider colorings	\(\f \colon \R^n~\to~\cC \subset \R\)  that are \emph{somewhere comeager},  and
	conditions \(\cQ\) for which all points of the space are admissible centers
	\((\mC = \R^n)\) and for which the set \(\r\) of admissible radii is \emph{comeager near
		\(0\)}.
	
	\begin{remark}
		To streamline notation, we write
		\(\cQ_{\cPx,\r^\epsilon}^* = \cQ_{\cC,\cPx,\R^n,\r}\) to indicate that the above conditions
		are satisfied. Accordingly, the statement that \emph{\(\f\) satisfies
			\(\cQ_{\cPx,\r^\epsilon}^*\)} means that:
		\begin{enumerate}
			\item \(\f \colon \R^n \to \cC\) is a somewhere comeager coloring.
			\item \(\f\) satisfies \(\cQ_{\cC,\cPx,\R^n,\r}\) with \(\r\)
			comeager in \((0,\epsilon)\).
		\end{enumerate}
	\end{remark}
	
	\begin{definition}
		A property \(\cPx\) is a \emph{\ucp} if there exists \(\delta \in (0, 1)\) such that for every admissible sphere \(S= S_r(p)\) and every \(x \in S\), the open cap \(B_{\delta r\sqrt{2}}(x) \cap S_r(p)\) has the property \(\cPx\). A property \(\cPx\) is a \emph{countable {\ucp}} if the open cap \(B_{\delta r\sqrt{2}}(x) \cap S_r(p)\) contains a countable subset that has property \(\cPx\).
	\end{definition}
	
	The definition above states that a property \(\cPx\) is a {\ucp} if every open hemisphere \(B_{r\sqrt{2}}(x) \cap S_r(p)\) of an admissible sphere \(S_r(p)\) 
	contains a proper cap - of radius \(\delta r\sqrt{2}\) - having property \(\cPx\), and the relative size (radius) of that cap is uniformly bounded on all 
	admissible spheres.
	
	The main result of this section is the following theorem.
	
	\begin{theorem}\label{thm:general_formQ}
		Let \(\f\) be a coloring that satisfies \(\cQ_{\cPx,\r^\epsilon}^*\).
		\begin{enumerate}
			\item If every comeager subset of any admissible sphere has property \(\cPx\), then \(\f\) is
			somewhere locally monochromatic.
			\item If \(\f\) is somewhere locally monochromatic and \(\cPx\) is a {\ucp}, then \(\f\) is monochromatic.
		\end{enumerate}
	\end{theorem}

	\begin{proof}
		(1) Since \(\f\) is somewhere comeager, there exist a color \(c \in \cC\) and a nonempty open
		set \(U \subset \R^n\) such that  \(\f^{-1}(c)\) is comeager in \(U\).  
		
		Let \(p \in U\) and \(R>0\) such that \(B_R(p) \subset U\). Then \(\f^{-1}(c)\) is comeager in \(B_R(p)\) 
		and from Lemma~\ref{lem:polarKU} we conclude that the set
		\[
		G_p = \{ r \in (0,R) \mid \f^{-1}(c) \text{ is comeager in } S_r(p) \}
		\]
		is comeager in \((0,R)\). 
		
		Let \(\delta = \min(R,\epsilon)\).
		Both \(G_p\) and \(\r \) are comeager in
		\((0,\delta)\), so their intersection,
		\[
		A_p = \{ r\in (0,\delta)\cap \r \mid \f^{-1}(c) \text{ is comeager in } S_r(p)\}
		\]
		is also comeager in \((0,\delta)\), hence nonempty. 
		Thus, there exists an \emph{admissible} radius \(r\) such that the set 
		\(\f^{-1}(c) \) is comeager in \(S_r(p)\), an admissible sphere.
		By hypothesis, this comeager set has property \(\cPx\), and since \(\f\) satisfies
		\(\cQ_{\cPx,\r^\epsilon}^*\), it follows that \(\f(p) = c\).
		As this holds for every \(p \in U\), the coloring \(\f\) is constant on the open
		set \(U\), and hence \(\f\) is somewhere locally monochromatic.
		
		(2) Suppose that \(\f\) is somewhere locally monochromatic.
		Then there exists a monochromatic open nonempty ball \(B_{r_0}(p)\).
		Let
		\[
		R = \sup \{ r > 0 \mid B_r(p) \text{ is monochromatic} \}.
		\]
		Clearly \(R \geqslant r_0 > 0\).
		
		Assume, for contradiction, that \(R < \infty\). Let \(q\in S_R(p)\). For a small enough admissible radius \(r>0\), the intersection \(S_r(q) \cap B_R(p)\) is as close to a hemisphere on \(S_r(q)\) as we want - in particular, we can choose \(r\) so that the intersection contains an open cap of radius \(\frac{1+\delta}{2} r \sqrt{2}\). Keeping \(r\) fixed and moving \(q\) slightly away from \(p\) would yield a bit smaller intersection \(S_r(q) \cap B_R(p)\), but we can still keep it as close to a hemisphere on \(S_r(q)\) as we want - to contain an open cap of radius \(\delta r\sqrt{2}\)  - by staying close enough, \(d(p,q) < R+t_0\) for some suitably small $t_0>0$. Then \(S_r(q)\) has a monochromatic open cap of radius \(\delta r \sqrt{2}\), and that cap has property \(\cPx\), therefore, \(f(q) = f(p)\). We conclude that \(B_{R+t_0}(p)\) is monochromatic, contradicting the definition of \(R\) as the supremum.

		Therefore, \(R = \infty\), and \(\f\) is monochromatic on \(\R^n\).
	\end{proof}
	
	Theorem~\ref{thm:general_formQ} has the following consequence for rigid geometric properties.
	
	\begin{theorem}\label{thm:rigid_Q}
		Let \(\f\) be a coloring that satisfies \(\cQ_{\cPx,\r^\epsilon}^*\) for a  rigid geometric property \(\cPx\).
		\begin{enumerate}
			\item If every admissible sphere contains a countable subset  that has property \(\cPx\), then \(\f\) is somewhere locally monochromatic.
			\item  If \(\cPx\) is a countable {\ucp}, then \(\f\) is monochromatic.
		\end{enumerate}
	\end{theorem}
	
	\begin{proof}
		Let \(S=S_R(p)\) be an admissible sphere and let \(A\subset S\) be any comeager subset. Let \(\Gamma= \{q_i\}_{i\in I}\) be a countable subset of \(S\) 
		that has property \(\cPx\). For \(i\in I\) let 
		\(\phi_i \colon G \to S\), \( \phi_i({a}) = a\cdot q_i\),  as in Lemma~\ref{lem:cont_open}. Since \(\phi_i\) is continuous, open, and surjective, and \(A\) is comeager in \(S\), it follows that \(\phi_i^{-1}(A)\) is comeager in \(G\), for all \(i \in I\). As \(I\) is countable, the intersection of these sets is comeager, hence, there exists
		\[
		a \in \bigcap_{i \in I} \phi_{i}^{-1}(A) \subset G,
		\]
		and then \({a}\cdot \Gamma \subset A\). Since \(\cPx\) is a rigid geometric property, \(A\) has property \(\cPx\). By the first part of Theorem~\ref{thm:general_formQ}, \(\f\) is somewhere locally monochromatic.
		
		The hypothesis of the second statement implies the condition in the first one, hence \(\f\) is somewhere locally monochromatic. Since the group of rigid motions of a sphere acts transitively and \(\cPx\) is a rigid geometric property, those hypotheses also imply that the condition of the second part of Theorem~\ref{thm:general_formQ} are satisfied. Therefore, \(\f\) is monochromatic.
	\end{proof}
	
	\begin{corollary}\label{cor:open_subsets}
		Let \(\f\) be a coloring that satisfies \(\cQ_{\cPx,\r^\epsilon}^*\) for a  rigid geometric property \(\cPx\). If every open set of every sphere has a countable subset with property \(\cPx\), then \(\f\) is monochromatic.
	\end{corollary}
	
	\medskip
	
	\section{Simplex Shape Conditions}\label{sec:IT-ET-RT}
	
	In this section we study colorings of \(\R^n\) satisfying \(\cQ_{\cPx, \r^{\epsilon}}^*\), where \(\cPx\) is one of the  rigid geometric properties \(\cPx_{(IS^m)}\), \(\cPx_{(ES^m)}\), or \(\cPx_{(RS^m)}\), with \(2\leqslant m \leqslant n\).
	
	The main result of this section is the following.
	
	\begin{theorem}\label{thm:BPdisk}
		Let \(\f\) be a coloring or \(\R^n\) that satisfies a condition \(\cQ_{\cPx, \r^{\epsilon}}^*\),
		where \(\cPx\) is one of \(\cPx_{(IS^m)}\), \(\cPx_{(ES^m)}\), or \(\cPx_{(RS^m)}\), with \(2\leqslant m \leqslant n\).
		\begin{enumerate}
			\item If \(m<n\), or \(m=n\) and \(\cPx =  \cPx_{(IS^n)}\), or \(m=n\geqslant 3\) and \(\cPx = \cPx_{(RS^n)}\), then \(\f\) is monochromatic.
			\item If \(m=n\) and \(\cPx=\cPx_{(ES^n)}\), or \(m=n=2\) and \(\cPx=\cPx_{(RS^2)}\), then \(\f\) is somewhere monochromatic, but not necessarily monochromatic.
		\end{enumerate}
	\end{theorem}
	
	\begin{proof}
		Every sphere in \(\R^m\) contains the vertices of isosceles, equilateral, and right \(m-\)simplices. Since every sphere in \(\R^n\) contains a subset isometric with a sphere in \(\R^m\), it follows that every sphere in \(\R^n\) contains the vertices of isosceles, equilateral, and right \(m-\)simplices. By the first part of Theorem~\ref{thm:rigid_Q}, the coloring \(\f\) is somewhere monochromatic.
		
		Suppose \(m<n\). Let \(S\) be a sphere in \(\R^n\), let \(U \subset S\) be a nonempty open subset of \(S\), and let \(p \in U\). Then there exists an open ball  \(B_{r}(p)\) in \(\R^n\) such that \(p \in B_r(p) \cap S \subset U\). The intersection \(S_{r/2}(p)\cap S\) is isometric with  a sphere in \(\R^{n-1}\) and, since \(m \leqslant n-1\), it contains a subset isometric with a sphere \(S'\) in \(\R^m\). The sphere \(S'\) contains the vertices of isosceles, equilateral, and right \(m-\)simplices, and therefore, \(U\) contains the vertices of isosceles, equilateral, and right \(m-\)simplices. By Corollary~\ref{cor:open_subsets}, the coloring \(\f\) is monochromatic.
		
		Suppose that \(m=n\) and \(\cPx =  \cPx_{(IS^n)}\). Let \(U\), \(p\), and \(r\) be as above, and let \(\sigma'\) be any \((n-1)-\)simplex with vertex set \(V \subset S_{r/2}(p)\cap S\). Then the \(n-\)simplex with vertex set \(V \cup \{p\}\) is an isosceles \(n-\)simplex with vertices in \(U\). By Corollary~\ref{cor:open_subsets}, the coloring \(\f\) is monochromatic.
		
		Suppose that \(m=n \geqslant 3\) and \(\cPx = \cPx_{(RS^n)}\). Let \(p_0, p_1, \ldots, p_n\) be the vertices of a right simplex \(\sigma\) with apex \(p_0\) and such that \(d(p_0,p_i) = a> 0\) for all \(i=1, \ldots, n\). The circumradius of \(\sigma\) is \(r=a\sqrt{n}/2\). Therefore, if 
		\[1> \delta > \sqrt{\frac{2}{n}},\]
		then every open cap of radius \(\delta r \sqrt{2}\) of every sphere of radius \(r\) contains a finite set with property \(\cPx_{(RS^n)}\). By the second part of Theorem~\ref{thm:rigid_Q}, the coloring \(\f\) is monochromatic.
		
		Finally, we construct a (counter)example to show that if \(m=n\) and \(\cPx=\cPx_{(ES^n)}\), or if \(m=n=2\) and \(\cPx=\cPx_{(RS^2)}\), then \(\f\) does not have to be monochromatic. Define a coloring \(\f \colon \R^n \to \Z\) by horizontal strips:
		\[
		\f(x_1, \ldots, x_n)=\lfloor x_n\rfloor\in\mathbb Z.
		\]
		The color class of \(c\in\mathbb Z\) is
		\[
		\f^{-1}(c)=\R^{n-1}\times[c,c+1).
		\]
		Each \(\f^{-1}(c)\) is Borel, hence \(\f\) is a Baire coloring. Moreover, each color class is convex.
		
		For equilateral simplices, and for full-dimensional right simplices in \(\R^2\) (right-angle triangles), the circumcenter lies in the convex hull of
		the vertices. Since each color class is convex, if the vertices of any equilateral
		or right simplex lie in the same color class, then so does the circumcenter.
		Thus \(\f\) is a Baire coloring that satisfies \(\cQ_{(ES^{n}), \r^{\epsilon}}^*\) and, for \(n=2\), \(\cQ_{(RS^{2}), \r^{\epsilon}}^*\). However, \(\f\) is not monochromatic.
		
		The construction above can easily be modified to the case of finitely many
		colors. For \(|\cC| \geqslant 2\), finite, merge all color classes with \(c<0\) into one class and all color
		classes with \(c\geqslant  {|\cC|}-2\) into another, leaving the remaining classes unchanged.
		This yields a Baire coloring with \(|\cC|\) colors that satisfies both
		\(\cQ_{(ES^{n}), \r^{\epsilon}}^*\) and \(\cQ_{(RS^{2}), \r^{\epsilon}}^*\), but is not monochromatic.
	\end{proof}
	
	\medskip

	\section{Edge-Length Conditions}\label{sec:edge_length}
	
	In this section we study conditions based on simplices with restricted edge lengths. 
	
	Let \(\S \subset(0,\infty)\) be a nonempty set and \(1 \leqslant k \leqslant N = \binom{n+1}{2}\). 	A somewhere comeager coloring 	\(\f\colon \R^n \to \cC\) satisfies condition \(\cQ_{k \S N, \r^{\epsilon}}^*\) if for every color \(c\in \cC\) and every admissible sphere \(S_r(p)\), if \(\f^{-1}(c) \cap S_r(p)\) contains the vertices of an \(n-\)simplex \(\sigma\) with at least \(k\) edge-lengths in \(\S\),  then \(f(p) = c\).
	
	The following result characterizes the sets \(\S\) for which every coloring satisfying the condition  \(\cQ_{(k\S N), \r^{\epsilon}}^*\) is monochromatic.
	
	\begin{theorem}\label{thm:infL}
		Let \(1\leqslant k < N = \binom{n+1}{2}\). The following are equivalent:
		\begin{itemize}
			\item[(i)] Every coloring \(\f \colon \R^n \to \cC\) satisfying \(\cQ_{(k\S N), \r^{\epsilon}}^*\) is
			monochromatic.
			\item[(ii)] \(\inf(\S)=0\).
		\end{itemize}
	\end{theorem}
	
	\begin{proof}
		We first prove \(ii) \rightarrow i)\), by showing that the rigid geometric property \(\cPx_{((N-1)\S N)}\) satisfies the hypothesis of Corollary~\ref{cor:open_subsets} - then \(\cPx_{(k\S N)}\) satisfies the same hypothesis for all \(1\leqslant k \leqslant N-1\).
		
		Let \(S= S_0\) be \emph{any} sphere in \(\R^n\) (here 0 is the index, not the radius) and let \(U = U_0 \subset S_0=S\) be an open subset. Let 
		\(p_0 \in U_0\) and \(t_0  \in (0,\infty) \) 
		such that \(B_{t_0}(p_0) \cap S_0 \subset U_0\). Since \(\inf(\S)=0\), there exists \(r_0 \in \S \cap (0,t_0)\). Define \(S_1  = S_{r_0}(p_0) \cap S_0 \subset U_0\). Then \(d(p_0,q) = r_0 \in \S\) for all \(q\in S_1\) and \(S_1\) is isometric to a sphere of radius \(t_1 > 0\) in \(\R^{n-1}\). Let \(p_1 \in S_1\). 
		
		We repeat the construction above for \((p_1, S_1, U_1 = S_1, t_1)\) instead of \((p_0, S_0, U_0, t_0)\), and then iterate.  We construct points \(p_0, p_1, \ldots, p_{n-2}\) on a decreasing chain of spheres \(S_0 \supset S_1 \supset \dotsb \supset S_{n-2}\), such that the dimensions of the spheres decrease by 1 at each step, and \(d(p_i, p_j) = r_i \in \S \) for all \(0 \leqslant i < j \leqslant n-2\). We can repeat the construction one more time, using circles \(S_{n-2}\) and \(S_{r_{n-2}}(p_{n-2})\), and let \(S_{n-1} = \{p_{n-1}, p_n\}\) be the resulting intersection. 
		
		Then the \(n-\)simplex with vertex set \(\Gamma = \{p_0, p_1, \ldots, p_{n-1},  p_n\}\subset U\) has the edge-lengths, with the possible exception of \(d(p_{n-1},p_n)\), in the set \(\S\) of admissible edge-lengths. Hence \(\Gamma \subset U\) has property \(\cPx_{((N-1)\S N)}\). By Corollary~\ref{cor:open_subsets}, for every \(1\leqslant k < N\), every coloring \(\f \colon \R^n \to \cC\) satisfying \(\cQ_{(k\S N), \r^{\epsilon}}^*\) is monochromatic.
		
		Next we prove, by \emph{reductio ad absurdum}, that  \(i) \rightarrow ii)\).  Assume that \(i)\) holds but \(\inf(\S)>0\).
		Choose \(\delta >0\) such that \(\delta \sqrt{n} <\inf(\S)\), and partition \(\R^n\) into
		half-open hypercubes
		\[
		Q_{(x_1, x_2, \ldots, x_n)}=\bigtimes_{j=1}^{n} [\delta x_j, \delta (x_j+1)), \qquad (x_1, \ldots, x_n)\in\Z^n.
		\]
		
		Fix a bijection \(\kappa\colon\Z^n\to\N\), and assign to  each point
		\(p\in Q_{(x_1, \ldots, x_n)}\) the color \(\f(p)=\kappa(x_1, \ldots, x_n)\).
		Clearly \(f\) is a somewhere comeager coloring. 
		
		Any segment joining points of the same color has length smaller than \(\delta \sqrt{n} < \inf(\S)\). Hence, if an \(n-\)simplex has an 
		edge of length in \(\S\), then it cannot have all the vertices of the same color. 
		Then, for all \(1\leqslant k <N\), the coloring \(\f\) satisfies \(\cQ_{(k\S N), \r^{\epsilon}}^*\), but is not monochromatic, contradicting \(i)\). Hence, \(\inf(\S)=0\).
	\end{proof}
	
	The case \(k=N = \binom{n+1}{2}\)  exhibits a qualitatively different behavior, as shown in the following example.
	
	\begin{example}\label{exmp:not_inf}
		Let \(\S \in (0,\infty)\) be a countable set with \(\inf(\S) = 0\) - for example, the set of values of a decreasing sequence converging to 0. The cardinality of the set 
		\[
		E_n = \{ (i,j) \mid 0 \leqslant i < j \leqslant n\}
		\]
		is \(|E_n| = N = \binom{n+1}{2}\) and the cardinality of the set of functions \(h \colon E_n \mapsto \S\) is \(\aleph_0^N= \aleph_0\). 
		We call such a function \(h\) \emph{feasible} 
		if there exists an \(n-\)simplex in \(\R^n\), with vertex set \(\{p_0, p_1, \ldots, p_n\}\), such that for all \((i,j) \in E_n\) we have \(d(p_i,p_j) = h(i,j)\). Not all functions \(h\) are feasible - for example, some violate the triangle inequality. For a feasible \(h\), there exists a unique - up to isometries - \(n-\)simplex \(\sigma\) 
		in \(\R^n\) with those prescribed edge lengths. Therefore, the set of circumradii of simplices realizing the feasible functions \(h\) is countable. 
		
		Let \(\r\) be the complement in \((0,\infty)\) of that countable set - then \(\r\) is comeager, hence, comeager near 0. No admissible sphere contains the vertices of an \(n-\)simplex with all edge-lengths in \(\S\). Therefore, every coloring \(\f \colon \R^n \to \cC\) that is somewhere comeager satisfies condition  \(\cQ_{(N\S N), \r^{\epsilon}}^*\), hence that condition does not force a coloring to be monochromatic, and not even somewhere locally monochromatic. \qedhere
	\end{example}
	
	Thus, for \((N\S N)\), a sufficient condition for monochromaticity must be stronger than \(\inf(\S)=0\). We give such a condition below.
	
	\begin{theorem}\label{thm:comeagerL}
		If the set \(\S\) of admissible edge-lengths is comeager near 0, then every coloring \(\f \colon \R^n \to \cC\) satisfying \(\cQ_{(N\S N), \r^{\epsilon}}^*\) is
		monochromatic.
	\end{theorem}
	
	\begin{proof}
		Let \(\delta > 0\) such that \(\S\) is comeager in \((0,\delta)\).
		
		Let \(S\) be a sphere of radius \(R>0\) in \(\R^n\) and \(U \subset S\) an open subset of \(S\). Let \(p_0 \in U\) and \(t_0 > 0 \) such that \(p_0 \in S\cap B_{t_0}(p_0) \subset 
		U\). 
		
		Let \(h \colon (0,t_0) \to (0,\infty)\) be the function defined as follows.
		For \(t\in (0,t_0)\), the intersection \(S\cap S_{t}(p_0)\) is isometric with a sphere \(S' \subset \R^{n-1}\). Let \(h(t)\) be the length of the edges of a regular \((n-1)-\)simplex inscribed in \(S'\).  For small enough \(t_0 > 0\), the continuous function \( h \colon (0,t_0) \to (0,\infty)\) is increasing. Let \(r_0 \in (0,t_0)\) such that \(\max(r_0, h(r_0)) < \delta\) and \(h \colon (0,r_0) \to (0,h(r_0))\) is a homeomorphism. Since both \(\S\) and \(h(\S \cap (0,r_0))\) are comeager in \((0,h(r_0))\), it follows that their intersection is nonempty. Let \(r \in \S \cap (0, r_0)\) such that \(h(r) \in \S\).
		
		Let \(\{p_1, \ldots, p_n\}\) be the set of vertices of a regular \((n-1)-\)simplex inscribed in the sphere \(S\cap S_r(p_0) \subset U\). 
		Then the  \(n-\)simplex with set of vertices \(\{p_0, p_1, \ldots, p_{n-1},  p_n\}\subset U\) has the edge-lengths \(r, h(r) \in \S\). By Corollary~\ref{cor:open_subsets}, every coloring \(\f \colon \R^n \to \cC\) satisfying \(\cQ_{(N\S N), \r^{\epsilon}}^*\) is monochromatic.
	\end{proof}
	
	\medskip

	\section{Volume Conditions}\label{sec:volumes}
	
	In this section we study conditions based on simplices with restricted admissible volumes. 
	
	Let \(\V \subset(0,\infty)\) be a nonempty set and let \(2 \leqslant m \leqslant n\).  A somewhere comeager coloring \(\f\colon \R^n \to \cC\) satisfies condition \(\cQ_{(\vol_m, \V), \r^{\epsilon}}^*\) if for every color \(c\in \cC\) and every admissible sphere \(S_r(p)\), if \(\f^{-1}(c) \cap S_r(p)\) contains the vertices of an \(m-\)simplex \(\sigma\) with \(\vol_m(\sigma) \in \V\), then \(f(p) = c\).
	
	The following result characterizes the sets \(\V\) for which every coloring satisfying the condition \(\cQ_{(\vol_m, \V), \r^{\epsilon}}^*\) is monochromatic.
	
	\begin{theorem}\label{thm:infL_vol}
		Let \(2\leqslant m \leqslant n\). The following are equivalent:
		\begin{itemize}
			\item[(i)] Every coloring \(\f \colon \R^n \to \cC\) satisfying \(\cQ_{(\vol_m, \V), \r^{\epsilon}}^*\) is
			monochromatic.
			\item[(ii)] \(\inf(\V)=0\).
		\end{itemize}
	\end{theorem}
	
	\begin{proof}
		We first prove \(ii) \rightarrow i)\), by showing that the rigid geometric property \(\cPx_{(\vol_m, \V)}\) satisfies the hypothesis of the second part of Theorem~\ref{thm:rigid_Q}.
		
		Let \(S\) be a sphere of radius \(R>0\) in \(\R^n\) and \(U \subset S\) an open subset of \(S\). Let \(p_0 \in U\) and \(r_0 > 0 \) such that \(p_0 \in S\cap B_{r_0}(p_0) \subset  U\).  For  \(r \in (0,r_0)\), let \(\{p_1, p_2, \ldots, p_n\}\) be the vertices of a regular \((n-1)-\)simplex inscribed in \(S\cap S_{r}(p_0)\). Let \(\Delta = \Delta(r)\) be the \(m-\)simplex with vertices \(\{p_0, p_1, \ldots, p_m\}\), and let \(h \colon (0, r_0) \to (0, \infty)\), defined by \(h(r) = \vol_m(\Delta(r))\). 
		Fix \(r_1 \in (0,r_0)\). Since \(\inf(\V)=0\), there exists \(v \in \V\) such that \(0 < v < h(r_1)\). Since \(h\) is continuous and \(\lim_{r\to 0} h(r) = 0\), there exists \(r\in (0,r_1)\) such that \(h(r) = v\). Then the \(m-\)simplex \(\Delta(r)\) has all vertices in \(U\) and has \(m-\)volume in \(\V\).
		
		By Corollary~\ref{cor:open_subsets}, every coloring satisfying \(\cQ_{(\vol_m, \V), \r^{\epsilon}}^*\) is monochromatic.
		
		Next we prove, by \emph{reductio ad absurdum}, that  \(i) \rightarrow ii)\).  Assume that \(i)\) holds but \(\inf(\V)>0\).
		
		We extend the \(m-\)volume function to convex hulls of \(m+1\) points in \(\R^n\) by setting the volume to be 0 if the points are not affinely independent. Let \(\square\) be the closed unit hypercube in \(\R^n\). Let
		\(h \colon \square^{m+1} \to [0,\infty]\) be defined by \(h(p_0, \ldots, p_m) = \vol_m(\Delta_{p_0, \ldots, p_m})\). As \(h\) is continuous and \(\square^{m+1}\) is compact, there exists \(c_m > 0\) such that \(h(p_0, \ldots, p_m) \leqslant c_m\) for all \((p_0, \ldots, p_m)~\in~\square^{m+1}\). Then, for every \(\delta > 0\),  closed hypercube \(\square_{\delta} \subset \R^n\) of edge-length \(\delta\), and points \(p_0, \ldots, p_m~\in~\square_{\delta}\), we have 
		\(\vol_m(p_0, \ldots, p_m) \leqslant c_m \delta^m\). 
		
		Fix \(\delta > 0\) such that \(c_m\delta^m < \inf{(\V)}\). The coloring constructed in the second half of the proof of Theorem~\ref{thm:infL} satisfies the condition  \(\cQ_{(\vol_m, \V), \r^{\epsilon}}^*\) but is not monochromatic, contradicting \(i)\). Hence \(\inf(\V)= 0\).
	\end{proof}

\end{document}